\documentclass[12pt,twoside]{amsart}
\usepackage{amssymb,amsmath,amsthm, amscd, enumerate, mathrsfs, upgreek}
\usepackage{graphicx, hhline, tikz}
\usepackage[colorlinks=true,pagebackref,hyperindex]{hyperref}
\usepackage[all]{xy}
\usepackage{color}
\usepackage[backrefs, alphabetic, initials]{amsrefs}
\usepackage{color} 
\usepackage{tikz-cd}
\usepackage{latexsym}
\usepackage[T1]{fontenc} 
\usepackage{fancyhdr}
\usepackage{enumerate}
\usepackage{array, longtable}
\usepackage{caption}
\usepackage{mathtools}
\usepackage{cancel}

\newcolumntype{C}{>{$}c<{$}}
\newcolumntype{L}{>{$}l<{$}}

\title[Optimal upper bound for degrees of canonical $\mathbb{Q}$-Fano $3$-folds]
{Optimal upper bound for degrees of canonical Fano threefolds of Picard number one}

\date{\today, version 0.02}

\author{Chen Jiang}
\address{Chen Jiang, Shanghai Center for Mathematical Sciences \& School of Mathematical Sciences, Fudan University, Shanghai 200438, China}
\email{\href{chenjiang@fudan.edu.cn}{chenjiang@fudan.edu.cn}}
\urladdr{\href{https://chenjiangfudan.github.io/home/index.html}{https://chenjiangfudan.github.io/home/index.html}}

\author{Haidong Liu}
\address{Haidong Liu, Sun Yat-Sen University, School of Mathematics, Guangzhou 510275, China}
\email{\href{liuhd35@mail.sysu.edu.cn}{liuhd35@mail.sysu.edu.cn},\href{jiuguiaqi@gamil.com}{jiuguiaqi@gmail.com}}
\urladdr{\href{https://sites.google.com/view/liuhaidong}{https://sites.google.com/view/liuhaidong}}

\author{Jie Liu}
\address{Jie Liu, Institute of Mathematics, Academy of Mathematics and Systems Science, Chinese Academy of Sciences, Beijing 100190, China}
\email{\href{jliu@amss.ac.cn}{jliu@amss.ac.cn}}
\urladdr{\href{http://www.jliumath.com}{http://www.jliumath.com}}

\DeclareMathOperator{\Supp}{Supp}

\DeclareMathOperator{\rank}{rank}

\DeclareMathOperator{\mult}{mult}

\DeclareMathOperator{\reg}{reg}
\DeclareMathOperator{\Sing}{Sing}

\DeclareMathOperator{\Cl}{Cl}

\newcommand{\qW}{\text{\rm q}_{\text{\rm W}}}
\newcommand{\qQ}{\text{\rm q}_{\mathbb{Q}}}

\newcommand\lcm{{\text{l.c.m.}}}

\newtheorem{thm}{Theorem}[section]
\newtheorem{lem}[thm]{Lemma}
\newtheorem{prop}[thm]{Proposition}

\newtheorem{cor}[thm]{Corollary}

\theoremstyle{definition}
\newtheorem{ex}[thm]{Example}
\newtheorem{defn}[thm]{Definition}
\newtheorem{rem}[thm]{Remark}

\newtheorem{case}{Case}

\makeatletter
 
 \@addtoreset{equation}{section}
\makeatother

\begin{document}

\begin{abstract}
We show that for a $\mathbb Q$-factorial canonical Fano $3$-fold $X$ of Picard number $1$, $(-K_X)^3\leq 72$. The main tool is a Kawamata--Miyaoka type inequality which relates $(-K_X)^3$ with $\hat{c}_2(X)\cdot c_1(X)$, where $\hat{c}_2(X)$ is the generalized second Chern class.
 \end{abstract}

\maketitle 
\tableofcontents

\section{Introduction}\label{sec1}

Throughout this paper, we work over the complex number field $\mathbb C$. We will freely use the basic notation in \cites{KollarMori1998}.

A normal projective variety $X$ is called a {\it Fano variety} (resp. {\it weak Fano variety}) if the anti-canonical divisor $-K_X$ is ample (resp. nef and big). According to the minimal model program, Fano varieties with mild singularities form a fundamental class among research objects of birational geometry. Motivated by the classification theory of $3$-dimensional algebraic varieties, we aim to study the explicit geometry of Fano $3$-folds. 

Given a Fano $3$-fold $X$, we are interested in the {\it (anti-canonical) degree} $(-K_X)^3$ of $X$. This is an important invariant of Fano $3$-folds and it plays a key role in the classification of smooth Fano $3$-folds (see \cite{IskovskikhProkhorov1999}). On the other hand, the classification of canonical Fano $3$-folds is a wildly open problem and very few results are known. So it is quite crucial to understand the behavior of degrees of canonical Fano $3$-folds.

In this direction, Prokhorov \cite{Prokhorov2005} showed that for a Gorenstein canonical Fano $3$-fold $X$, $(-K_X)^3\leq 72$ and this bound is optimal. The first author and Yu Zou \cite{JiangZou2023} showed that for a canonical Fano $3$-fold $X$, $(-K_X)^3\leq 324$ (conjecturally this upper bound should be $72$). The main goal of this paper is to give an optimal upper bound for degrees of $\mathbb Q$-factorial canonical Fano $3$-folds of Picard number $1$.

\begin{thm}[{=Theorem~\ref{thm.sharpboundofdegree}}]
 Let $X$ be a $\mathbb Q$-factorial canonical Fano $3$-fold of Picard number $1$. Then $(-K_X)^3\leq 72$ and the equality holds if and only if $X\cong \mathbb P(1,1,1,3)$ or $\mathbb P(1,1,4,6)$.
\end{thm}

To prove this result, the main idea is to establish a Kawamata--Miyaoka type inequality which relates $(-K_X)^3$ with $\hat{c}_2(X)\cdot c_1(X)$, where $\hat{c}_2(X)$ is the generalized second Chern class (or orbifold second Chern class), see Theorem~\ref{thm.kmineqfor3fold} for details. We prove a more general Kawamata--Miyaoka type inequality for $\mathbb Q$-factorial $\epsilon$-lc Fano varieties of Picard number $1$, generalizing previous results \cites{Liu2019a,LiuLiu2023, LiuLiu2024}.

\begin{thm}\label{thm.KMinqEpsilonLC}
Let $0< \epsilon\leq 1$ be a real number. Let $X$ be a $\mathbb Q$-factorial $\epsilon$-lc Fano variety of dimension $n\geq 2$ and of Picard number $1$. Then we have
 \[
 c_1(X)^{n}<\frac{2(1+\epsilon)}{\epsilon} \hat{c}_2(X)\cdot c_1(X)^{n-2}.
 \]
\end{thm}

\section{Preliminaries}
Let $X$ be a normal variety of dimension $n$ such that its canonical divisor $K_X$ is $\mathbb Q$-Cartier. Then the \emph{Gorenstein index} $r_X$ of $X$ is defined as the smallest positive integer $m$ such that $mK_X$ is Cartier.
 
\subsection{Singularities}
Let $X$ be a normal variety such that $K_X$ is $\mathbb Q$-Cartier and let $f\colon Y\to X$ be a proper birational morphism. A prime divisor $E$ on $Y$ is called \emph{a divisor over} $X$. 
Write
\[
 K_Y=f^*K_X+\sum_Ea(E,X)E,
\]
where $a(E,X)\in \mathbb Q$ is called the \emph{discrepancy} of $E$. We say that $X$ has \emph{terminal} (resp. \emph{canonical}, \emph{klt}, \emph{$\epsilon$-lc} for some fixed $0\leq \epsilon\leq 1$) singularities if $a(E,X)>0$ (resp. $a(E,X)\geq 0$, $a(E,X)>-1$, $a(E,X)\geq \epsilon-1$) for any exceptional divisor $E$ over $X$.
Often we just simply say that $X$ is 
{\it terminal, canonical, klt, or $\epsilon$-lc}, respectively.



\subsection{Fano indices}\label{subsec.fi}

 Let $X$ be a klt weak Fano variety. 
 We can define 
 \begin{align*} 
 \qW(X) {}&:=\max\{q \mid -K_X \sim qB, \quad B\in \Cl (X) \};\\
 \qQ(X) {}&:=\max\{q \mid -K_X \sim_{\mathbb Q}qA, \quad A\in \Cl (X) \}. 
 \end{align*}
 We call $\qW(X)$ the \emph{Weil--Fano index} of $X$ and $\qQ(X)$ the \emph{$\mathbb Q$-Fano index} of $X$. 
 It is known that $\Cl (X)$ is a finitely generated Abelian group, so $\qW(X)$ and $\qQ(X)$ are positive integers and $\qW(X)\mid \qQ (X)$; moreover, $\qW(X)= \qQ (X)$ if $\Cl (X)$ is torsion-free. 
 For more details, see \cite[\S\,2]{IskovskikhProkhorov1999} or \cite{Prokhorov2010}.

Here we recall the following easy lemma.
\begin{lem}\label{lem ex torsion}
    Let $X$ be a klt weak Fano variety. Suppose that $p$ is a prime factor of $\qQ(X)$ such that $p\nmid \qW(X)$, then there is a torsion element in $\Cl (X)$ with order $p$.
\end{lem}
\begin{proof}
    Denote $q\coloneq \qQ(X)$. By assumption, we have $-K_X\sim_\mathbb{Q}qA$ for some Weil divisor $A$. Then $D\coloneq -K_X-qA$ is a torsion element in $\Cl(X)$. Denote by $s$ the torsion order of $D$. It suffices to show that $p\mid s$. Suppose that $p\nmid s$, then there exists an integer $k$ such that $p\mid ks+1$. This implies that $-K_X=qA+D\sim qA+(ks+1)D=p(\frac{q}{p}A+\frac{ks+1}{p}D)$ where $\frac{q}{p}A+\frac{ks+1}{p}D\in \Cl(X)$. This implies that $p\mid \qW(X)$ (see \cite[Proof of Lemma~3.2(1)]{Prokhorov2010}), a contradiction. 
 \end{proof}

\subsection{Reid's basket and Reid's Riemann--Roch formula}
Let $X$ be a canonical projective $3$-fold. According to Reid \cite[(10.2)]{Reid1987}, there is a collection of pairs of integers (permitting weights)
\[
 B_{X}=\{(r_{i}, b_{i}) \mid i=1, \cdots, s ; 0<b_{i}\leq\frac{r_{i}}{2} ; b_{i} \text{ is coprime to } r_{i}\}
\]
associated to $X$, called {\it Reid's basket}, where a pair $\left(r_{i}, b_{i}\right)$ corresponds to an orbifold point $Q_{i}$ of type $\frac{1}{r_{i}}\left(1, -1, b_{i}\right)$ which comes from deforming singularities of a terminalization of $X$ locally. 
Denote by $\mathcal{R}_X$ the collection of $r_i$ (permitting weights) appearing in $B_X$. 
Note that the Gorenstein index $r_X$ of $X$ is just $\lcm\{r_i\mid r_i\in \mathcal{R}_X\}$.

Recall that for a $\mathbb{Q}$-Cartier $\mathbb{Q}$-divisor $H$ on $X$, $c_2(X)\cdot H$ is defined as $c_2(Y)\cdot f^*H$ where $f\colon Y\to X$ is a resolution of singularities.

\begin{thm}[{\cite[(10.3)]{Reid1987}}]\label{thm.euler}
Let $X$ be a canonical Fano $3$-fold. Then
\begin{align}\label{eq.range}
 c_2(X)\cdot c_1(X) + \sum_{r_i\in \mathcal{R}_X} \left(r_i-\frac{1}{r_i}\right)=24\chi(X, \mathcal O_X)=24,
\end{align}
and 
\begin{align}\label{eq.RR-Fano}
 \frac{1}{2}c_1(X)^3+3- \sum_{(r_i,b_i)\in B_X}\frac{b_i(r_i-b_i)}{2r_i}=-\chi(X, \mathcal{O}_X(2K_X))= h^0(X,-K_X)\in \mathbb Z_{\geq 0}.
 \end{align}
\end{thm}

\begin{proof}
 As $-K_X$ is ample, by the Kawamata--Viehweg vanishing theorem, $h^0(X, \mathcal{O}_X)=\chi(X, \mathcal{O}_X)$ and $h^0(X,-K_X)=\chi(X, \mathcal{O}_X(-K_X) )$. Then the theorem follows from \cite[(10.3)]{Reid1987} and the Serre duality.
\end{proof}

\section{Kawamata--Miyaoka type inequality}\label{sec.upperbound}

\subsection{$\mathbb{Q}$-variant of Langer's inequality}

Let $X$ be a projective klt variety of dimension $n\geq 2$. Then there exists a closed subset $Z\subset X$ of codimension $3$ such that $X\setminus Z$ has only quotient singularities and admits a structure of a $\mathbb{Q}$-variety (see \cite[\S\,3.2 and \S\,3.6]{GrebKebekusPeternellTaji2019a}). Moreover, for any reflexive sheaf $\mathcal{E}$ on $X$, as explained in \cite[Construction~3.8 and \S\,3.7]{GrebKebekusPeternellTaji2019a}, we can define the \emph{generalized second Chern class} $\hat{c}_2(\mathcal{E})$, which is a symmetric $\mathbb{Q}$-multilinear map:
\[
 \hat{c}_2(\mathcal{E})\colon N^1(X)_{\mathbb{Q}}^{\times (n-2)} \longrightarrow \mathbb{Q},\quad (\alpha_1,\dots,\alpha_{n-2})\longmapsto \hat{c}_2(\mathcal{E})\cdot \alpha_1\cdots\alpha_{n-2}.
\]
On the other hand, since $X$ is $\mathbb{Q}$-factorial in codimension $2$, for any two $\mathbb{Q}$-divisor classes $\beta$ and $\gamma$ on $X$, the product $\beta\cdot \gamma$ is also a well-defined symmetric $\mathbb{Q}$-multilinear map:
\[
 \beta\cdot \gamma\colon N^1(X)_{\mathbb{Q}}^{\times (n-2)}\longrightarrow \mathbb{Q},\quad (\alpha_1,\dots,\alpha_{n-2})\longmapsto \beta\cdot \gamma\cdot \alpha_1\cdots\alpha_{n-2}.
\]
Now the \emph{$\mathbb{Q}$-Bogomolov discriminant} $\hat{\Delta}(\mathcal{E})$ of a reflexive sheaf $\mathcal{E}$ of rank $r$ is defined as follows:
	\[
	\hat{\Delta}(\mathcal{E})\coloneqq 2r \hat{c}_2(\mathcal{E}) - (r-1) c_1(\mathcal{E})^2,
	\]
which is viewed as a symmetric $\mathbb{Q}$-multilinear map $N^1(X)_{\mathbb{Q}}^{\times (n-2)}\rightarrow \mathbb{Q}$. According to \cite[Lemma~6.5]{KeelMatsukiMcKernan1994}, for any semistable reflexive sheaf $\mathcal{E}$ with respect to a collection of nef classes $(\alpha_1,\dots,\alpha_{n-1})$, the following \emph{$\mathbb{Q}$-Bogomolov--Gieseker inequality} holds:
\begin{equation}\label{eq.BGineq}
 \hat{\Delta}(\mathcal{E})\cdot \alpha_1\cdots\alpha_{n-2} \geq 0.
\end{equation}

\begin{lem}
	\label{lem.LangerIneqI}
	Let $\mathcal{E}$ be a reflexive sheaf of rank $r$ on a projective klt variety $X$ of dimension $n\geq 2$. Let $(\alpha_1,\dots,\alpha_{n-1})$ be a collection of nef classes such that $\alpha_1\cdots \alpha_{n-1}$ is not numerically trivial. Let $0=\mathcal{E}_0\subsetneq \mathcal{E}_1\subsetneq \dots\subsetneq \mathcal{E}_m=\mathcal{E}$ be the Harder--Narasimhan filtration of $\mathcal{E}$ with respect to $(\alpha_1,\dots,\alpha_{n-1})$. Denote by $\mathcal{F}_i$ the reflexive hull of the quotient $\mathcal{E}_i/\mathcal{E}_{i-1}$ and set $r_i=\rank(\mathcal{F}_i)$ for any $i$. Then we have
	\begin{equation}
		\label{eq.LangerIneqI}
	\hat{\Delta}(\mathcal{E})\cdot \alpha_1\cdots\alpha_{n-2} \geq - \sum_{1\leq i<j\leq m} r_i r_j \left(\frac{c_1(\mathcal{F}_i)}{r_i}-\frac{c_1(\mathcal{F}_j)}{r_j}\right)^2\cdot \alpha_1\cdots\alpha_{n-2}.
	\end{equation}
\end{lem}

\begin{proof}
	Since $\mathcal{E}$ is reflexive and any $\mathcal{E}_i$ is saturated in $\mathcal{E}$, each $\mathcal{E}_i$ is reflexive and we have the following exact sequence of reflexive sheaves
	\[
	0\longrightarrow \mathcal{E}_{i-1} \longrightarrow \mathcal{E}_i \longrightarrow \mathcal{F}_i.
	\]
	Since $\mathcal{E}_i\rightarrow \mathcal{F}_i$ is surjective in codimension $1$, we have $c_1(\mathcal{E}_i)=c_1(\mathcal{E}_{i-1}) + c_1(\mathcal{F}_i)$. On the other hand, thanks to \cite[Lemma~2.3]{Kawamata1992a} and \cite[Theorem~3.13]{GrebKebekusPeternellTaji2019a}, we obtain
	\[
	\hat{c}_2(\mathcal{E}_i)\cdot \alpha_1 \cdots \alpha_{n-2} \geq \left(\hat{c}_2(\mathcal{E}_{i-1}) + \hat{c}_2(\mathcal{F}_i) + c_1(\mathcal{E}_{i-1})\cdot c_1(\mathcal{F}_i)\right)\cdot \alpha_1\cdots\alpha_{n-2}.
	\]
	Now a straightforward computation derives
	\[
	\frac{\hat{\Delta}(\mathcal{E})\cdot \alpha_1\cdots \alpha_{n-2}}{r} \geq \sum_{i=1}^{m} \frac{\hat{\Delta}(\mathcal{F}_i)\cdot \alpha_1\cdots \alpha_{n-2}}{r_i} - \frac{1}{r}\sum_{i<j} r_i r_j \left(\frac{c_1(\mathcal{F}_i)}{r_i}-\frac{c_1(\mathcal{F}_j)}{r_j}\right)^2\cdot \alpha_1\cdots\alpha_{n-2}
	\]
	and hence the $\mathbb{Q}$-Bogomolov--Gieseker inequality \eqref{eq.BGineq} yields the desired inequality as $\mathcal{F}_i$ is semistable with respect to $(\alpha_1,\dots,\alpha_{n-1})$ for each $i$.
\end{proof}

The following $\mathbb{Q}$-variant of \cite[Theorem~5.1]{Langer2004} is more or less well-known to experts. We provide a proof for the reader's convenience.

\begin{thm}
	\label{thm.LangerIneqII}
	Let $\mathcal{E}$ be a reflexive sheaf of rank $r$ on a projective klt variety $X$ of dimension $n\geq 2$ and let $(\alpha_1,\dots,\alpha_{n-1})$ be a collection of nef classes. Then we have
	\[
	(\alpha_1\cdots \alpha_{n-2} \cdot \alpha_{n-1}^2) \cdot(\hat{\Delta}(\mathcal{E})\cdot \alpha_1\cdots\alpha_{n-2}) + r^2 (\mu_{\max} - \mu) (\mu-\mu_{\min}) \geq 0,
	\]
	where $\mu$ (resp. $\mu_{\max}$ and $\mu_{\min}$) is the slope (resp. maximal slope and minimal slope) of $\mathcal{E}$ with respect to $(\alpha_1,\dots,\alpha_{n-1})$.
\end{thm}

\begin{proof}
	We will follow the notation in Lemma~\ref{lem.LangerIneqI}. Without loss of generality, we shall assume $d\coloneqq \alpha_1\cdots \alpha_{n-2}\cdot \alpha_{n-1}^2>0$. Then combining the Hodge index theorem and Lemma~\ref{lem.LangerIneqI} yields
	\[
		\hat{\Delta}(\mathcal{E})\cdot \alpha_1\cdots\alpha_{n-2} 
		 \geq - \frac{1}{d} \sum_{1\leq i<j\leq m} r_i r_j (\mu_i-\mu_j)^2,
	\]
	where $\mu_i$ is the slope of $\mathcal{F}_i$ with respect to $(\alpha_1,\dots,\alpha_{n-1})$. As $r\mu=\sum r_i\mu_i$ and $\mu_{\max}\geq \mu_i\geq \mu_{\min}$ for any $1\leq i\leq m$, we obtain
	\begin{align*}
		r^2(\mu_{\max}-\mu) (\mu-\mu_{\min})
		 & = \left(\sum_{i=1}^{m} r_i (\mu_{\max} - \mu_i)\right) \left(\sum_{j=1}^{m} r_j(\mu_j - \mu_{\min})\right) \\
		 & \geq \sum_{1\leq i < j\leq m} r_i r_j (\mu_{\max}-\mu_i) (\mu_j-\mu_{\min}) \\
		 & \geq \sum_{1\leq i<j\leq m} r_i r_j (\mu_j - \mu_i)^2,
	\end{align*}
	which implies immediately the desired inequality.
\end{proof}

\subsection{Foliations}

We gather some basic notions and facts regarding foliations on varieties. We refer the reader to \cite[\S\,3]{Druel2021} and the references therein for a more detailed explanation. 

\begin{defn}
	A {\it foliation} on a normal variety $X$ is a non-zero coherent subsheaf $\mathcal{F}$ of the tangent sheaf $\mathcal{T}_X$ such that 
	\begin{enumerate}
		\item $\mathcal{T}_X/\mathcal{F}$ is torsion-free, and
		
		\item $\mathcal{F}$ is closed under the Lie bracket.
	\end{enumerate} 
	The {\it canonical divisor} of a foliation $\mathcal{F}$ is any Weil divisor $K_{\mathcal{F}}$ on $X$ such that $\det(\mathcal{F})\cong \mathcal{O}_X(-K_{\mathcal{F}})$. The \emph{rank} of $\mathcal{F}$ is defined as the generic rank of $\mathcal{F}$.
\end{defn}

Let $X_{\circ}\subset X_{\reg}$ be the largest open subset over which $\mathcal{T}_X/\mathcal{F}$ is locally free. A \emph{leaf} of $\mathcal{F}$ is a maximal connected and immersed holomorphic submanifold $L\subset X_{\circ}$ such that $\mathcal{T}_L=\mathcal{F}|_L$. A leaf is called \emph{algebraic} if it is open in its Zariski closure and a foliation $\mathcal{F}$ is said to be \emph{algebraically integrable} if its leaves are algebraic. 

Let $\varphi\colon Y\dashrightarrow X$ be a rational dominant map between normal varieties and let $\mathcal{F}$ be a foliation on $X$. Let $Y_{\circ}\subset Y_{\reg}$ be the largest open subset such that $\varphi$ is well-defined and $\varphi(Y_{\circ})\subset X_{\reg}$. The \emph{pull-back} $\varphi^{-1}\mathcal{F}$ of $\mathcal{F}$ is the unique foliation on $Y$ such that 
\[
(\varphi^{-1}\mathcal{F})|_{Y_{\circ}}=\ker(d\varphi|_{Y_{\circ}}\colon \mathcal{T}_{Y_{\circ}} \rightarrow \mathcal{T}_{X}/\mathcal{F}).
\]

Let $\mathcal{F}$ be an algebraically integrable foliation on a normal projective variety $X$. Then there exists a 
diagram, called the \emph{family of leaves}, as follows (\cite[\S\,3.6]{Druel2021}):
\begin{equation}
	\label{eq.familyofleaves}
	\begin{tikzcd}[row sep=large, column sep=large]
		U \arrow[r,"e"] \arrow[d,"p"]
		& X \\
		T
	\end{tikzcd}
\end{equation}
where $U$ and $T$ are normal projective varieties, the \emph{evaluation morphism} $e$ is birational and $p$ is an equidimensional fibration such that the image $e(p^{-1}(t))$ is the closure of a leaf of $\mathcal{F}$ for general $t\in T$. Assume in addition that $K_{\mathcal{F}}$ is $\mathbb{Q}$-Cartier, then there exists a canonically defined effective $e$-exceptional $\mathbb{Q}$-divisor $\Delta$ such that 
\begin{equation}
	\label{eq.PullBackKF}
	K_{e^{-1}\mathcal{F}} + \Delta \sim_{\mathbb{Q}} e^*K_{\mathcal{F}}.
\end{equation}

\begin{lem}[\protect{cf. \cite[Proposition 4.17]{Druel2021}}]
	\label{lem.Exc-div-family-leaves}
	Let $g\colon Y\rightarrow B$ be a surjective morphism between normal varieties with irreducible general fibers. Assume that there exists a birational projective morphism $f\colon Y\rightarrow X$ to a quasi-projective normal variety $X$ such that
	\begin{enumerate}
		\item the canonical divisor $K_{\mathcal{F}}$ of the foliation $\mathcal{F}$ on $X$ induced by $g$ is $\mathbb{Q}$-Cartier, and
		
		\item the restriction $f|_F\colon F \rightarrow f(F)$ to general fibers $F$ of $g$ is birational and finite.
	\end{enumerate}
 Write $K_{f^{-1}\mathcal{F}} + \Delta \sim_{\mathbb{Q}} f^*K_{\mathcal{F}}$ for some $f$-exceptional $\mathbb{Q}$-divisor $\Delta$. Then for any $f$-exceptional $g$-horizontal prime divisor $E$, we have $E\subset \Supp(\Delta)$.
\end{lem}

\begin{proof}
	Let $r\coloneqq \dim (Y)-\dim(B)$. We will follow the argument of \cite[Proposition 4.17]{Druel2021} by induction on $r$ with some modifications in our situation. If $r=1$, then the result follows from the first part of the proof of \cite[Proposition 4.17]{Druel2021}. Here we remark that though $X$ is assumed to be projective therein, but the proof is local and hence it works also in our situation.
	
	For $r\geq 2$, choose a normal projective compactification $\bar{X}$ of $X$ and let $\bar{\mathcal{F}}$ be the foliation on $\bar{X}$ induced by $\mathcal{F}$. Then $\bar{\mathcal{F}}$ is algebraically integrable. Denote by $ {p}\colon U\rightarrow T$ the family of leaves of $\bar{\mathcal{F}}$. By the assumption, after shrinking $X$ and accordingly $Y$, we may assume that there exists an open subset $B^{\circ}$ of $B$ such that there exist embeddings $B^{\circ}\hookrightarrow T$ and $Y^{\circ}\coloneqq Y\times_B  B^{\circ}\hookrightarrow U$ such that the following diagram commutative:
	\[
	\begin{tikzcd}[row sep=large,column sep=large]
		B  
		& Y \arrow[l,"{g}" above] \arrow[r,"{f}"] 
		& X \\
		B^{\circ} \arrow[d,hookrightarrow] \arrow[u,hookrightarrow]
		 & Y^{\circ} \arrow[l,"{g^{\circ}\coloneqq g|_{Y^{\circ}}}" above] \arrow[r,"{f^{\circ}\coloneqq f|_{Y^{\circ}}}"] \arrow[u,hookrightarrow]
 \arrow[d,hookrightarrow]
		 & X \arrow[d,hookrightarrow] \arrow[u,equal] \\
		T 
		 & U \arrow[l,"{p}" above] \arrow[r,"e"]
		 & \bar{X}.
	\end{tikzcd}
	\]
	Without loss of generality, we may assume that $\bar{X}\subset \mathbb{P}^N$ for some $N$. Let $\bar{E}$ be the closure of $E^{\circ}\coloneqq E\cap Y^{\circ}$ in $U$. Let $\bar{H}$ be a general hyperplane section of $\bar{X}$ and let $H\coloneqq \bar{H}\cap X$. Set $\bar{\mathcal{G}}\coloneqq \bar{\mathcal{F}}|_{\bar{H}}\cap T_{\bar{H}}$ and $\mathcal{G}\coloneqq \bar{\mathcal{G}}|_H$. Then $\bar{\mathcal{G}}$ and $\mathcal{G}$ are foliations of rank $r-1$. Denote $e^{-1}(\bar{H})$ by $\bar{G}$. Let $G^{\circ}\coloneqq \bar{G}\cap Y^{\circ}$ and $G\coloneqq f^{-1}(H)$. Then we get the following commutative diagram:
	\[
	\begin{tikzcd}[row sep=large,column sep=large]
		B  
		& G \arrow[l,"{\psi\coloneqq g|_G}" above] \arrow[r,"{\gamma\coloneqq f|_G}"] 
		& H \\
		B^{\circ} \arrow[d,hookrightarrow] \arrow[u,hookrightarrow]
		& G^{\circ} \arrow[l,"{\psi^{\circ}\coloneqq g|_{G^{\circ}}}" above] \arrow[r,"{\gamma^{\circ}\coloneqq f|_{G^{\circ}}}"] \arrow[u,hookrightarrow]
		\arrow[d,hookrightarrow]
		& H \arrow[d,hookrightarrow] \arrow[u,equal] \\
		T 
		& \bar{G} \arrow[l,"{\bar{\psi}=p|_{\bar{G}}}" above] \arrow[r,"{\bar{\gamma}=e|_{\bar{G}}}"]
		& \bar{H}.
	\end{tikzcd}
	\]	
	
	Let $\bar{F}$ be a fiber of $\bar{\psi}$. Since $\bar{E}$ dominates $T$, then $\bar{E}\cap \bar{F}$ has dimension $\geq r-1$. In particular, since $e$ is finite over the fibers of $\bar{\psi}$ and $\bar{H}$ is very ample, the intersection $\bar{E}|_{\bar{G}}$ is a non-zero divisor and we may assume that the following holds:
	\begin{enumerate}
		\item $\bar{E}|_{\bar{G}}$ (resp. $E|_{G}$ and $\Delta|_G$) is $\bar{\gamma}$ (resp. $\gamma$)-exceptional;
		
		\item Every irreducible components of $\bar{E}|_{\bar{G}}$ (resp $E|_G$) dominates $T$ (resp. $B$);
		
		\item For a general fiber $\bar{F}$ of $\bar{\psi}$, the intersection $\bar{G}\cap \bar{F}_{\reg}$ is smooth and $\bar{G}\cap \bar{F}_{\textrm{sing}}$ has codimension at least $3$ in $\bar{F}$;
		
		\item  $E\cap G^{\circ}=\bar{E}\cap G^{\circ}$ is non-empty.
	\end{enumerate}
	Now apply \cite[Proposition 3.6(3) and (1)]{Druel2021} to $(\bar{X},\bar{H},\bar{\mathcal{F}})$ and $(U,\bar{G},e^{-1}\bar{\mathcal{F}})$, respectively, yielding
	\[
	K_{\bar{\mathcal{G}}} \sim K_{\bar{\mathcal{F}}}|_{\bar{H}} + \bar{H}|_{\bar{H}}, \quad \text{and} \quad
	K_{\bar{\gamma}^{-1}\bar{\mathcal{G}}} \sim K_{e^{-1}\bar{\mathcal{F}}}|_{\bar{G}} + \bar{G}|_{\bar{G}} - B_{\bar{G}}
	\]
	for some effective $\bar{\gamma}$-exceptional divisor $B_{\bar{G}}$ on $\bar{G}$.  Then combing the condition (3) with \cite[Proposition 3.6(1)]{Druel2021} shows that $\Supp(B_{\bar{G}})$ does not dominate $T$. On the other hand, as $K_{f^{-1}\mathcal{F}}+\Delta \sim_{\mathbb{Q}} f^*K_{\mathcal{F}}$, by restricting to $H$ and $G^{\circ}$, we obtain
	\[
	K_{(\gamma^{\circ})^{-1}\mathcal{G}} \sim_{\mathbb{Q}} (\gamma^{\circ})^* K_{\mathcal{G}} - \Delta|_{G^{\circ}} - B_{G^{\circ}},
	\]
	where $B_{G^{\circ}}=B_{\bar{G}}|_{G^{\circ}}$. So we can find a $\gamma$-exceptional divisor $D$ supported on $G\setminus G^{\circ}$ such that
	\[
	K_{\gamma^{-1}\mathcal{G}} \sim_{\mathbb{Q}} \gamma^* K_{\mathcal{G}} - \Delta|_G - B_{G} + D,
	\]
	where $B_G$ is the strict transform of $B_{G^{\circ}}$ on $G$. By the conditions (1), (2) and the induction assumption, we obtain 
	\[
	\Supp(E|_G) \subset \Supp(\Delta|_G + B_G - D).
	\]
	However, note that $\Supp(B_G)$ and $\Supp(D)$ do not dominate $B$, we thus must have
	\[
	\Supp(E|_G)\subset \Supp(\Delta|_G).
	\]
	Hence $\Supp(E)\subset \Supp(\Delta)$ by the condition (4) and the generality of $H$.
\end{proof}

	\begin{prop}
		\label{prop.ExcDivFamLeaves}
		Let $\mathcal{F}$ be an algebraically integrable foliation on a normal projective variety $X$ such that $K_{\mathcal{F}}$ is $\mathbb{Q}$-Cartier. Let $e\colon U\rightarrow X$ be the evaluation morphism given in Diagram \eqref{eq.familyofleaves} and write $K_{e^{-1}\mathcal{F}}+\Delta \sim_{\mathbb{Q}} e^*K_{\mathcal{F}}$ as in \eqref{eq.PullBackKF}. Then for any $p$-horizontal $e$-exceptional prime divisor $E$, we have $\mult_E(\Delta)\geq 1$.
	\end{prop}
	
	\begin{proof}
 Since $K_{\mathcal{F}}$ is $\mathbb{Q}$-Cartier, we can find an open subset $X^{\circ}$ of $\eta_{e(E)}$ such that $mK_{\mathcal{F}}|_{X^{\circ}}\sim 0$ for some positive integer $m$, where $\eta_{e(E)}$ is the generic point of $e(E)$. Denote by $f^{\circ}\colon X_1^{\circ}\rightarrow X^{\circ}$ the quasi-\'etale cyclic cover associated to $mK_{\mathcal{F}}|_{X^{\circ}}$ (see \cite[\S\,3.6]{Druel2021}). Let $U_1^{\circ}$ be the normalization of the fiber product $U^{\circ}\times_{X^{\circ}} X_1^{\circ}$, where $U^{\circ}=e^{-1}(X^{\circ})$. Denote by $e^{\circ}_1\colon U_1^{\circ}\rightarrow X_1^{\circ}$ and $g^{\circ}\colon U_1^{\circ} \rightarrow U^{\circ}$ the natural morphisms, which satisfy the following commutative diagram:
		\[
		\begin{tikzcd}[row sep=large,column sep=large]
			U_1^{\circ} \arrow[d,"g^{\circ}" left] \arrow[r,"e_1^{\circ}"] 
			& X_1^{\circ} \arrow[d,"f^{\circ}"] \\
			U^{\circ} \arrow[r,"e^{\circ}\coloneqq e|_{U^{\circ}}"]
			& X^{\circ}
		\end{tikzcd}
		\]
		For simplicity, we set $\mathcal{F}^{\circ}\coloneqq \mathcal{F}|_{X^{\circ}}$, $\mathcal{F}_1^{\circ}=(f^{\circ})^{-1} \mathcal{F}^{\circ}$, $\mathcal{G}^{\circ}=(e^{\circ})^{-1}(\mathcal{F}^{\circ})$ and $\mathcal{G}_1^{\circ}=(e_1^{\circ})^{-1}(\mathcal{F}_1^{\circ})$. As explained in \cite[\S\,3.6]{Druel2021}, there exists an effective $e_1^{\circ}$-exceptional divisor $\Delta_1^{\circ}$ on $U_1^{\circ}$ such that
		\[
		K_{\mathcal{G}_1^{\circ}} + \Delta_1^{\circ} \sim (e_1^{\circ})^* K_{\mathcal{F}_1^{\circ}}.
		\]
		On the other hand, since $f^{\circ}$ is quasi-\'etale, we have $K_{\mathcal{F}_1^{\circ}} = (f^{\circ})^* K_{\mathcal{F}^{\circ}} \sim 0$.	So $\Delta_1^{\circ}$ is an integral effective Weil divisor. Let $E_1^{\circ}$ be a codimension $1$ irreducible component of $(g^{\circ})^{-1}(E^{\circ})$, where $E^{\circ}\coloneqq E\cap U^{\circ}$. Denote by $r$ the ramification index of $g^{\circ}$ along $E_1^{\circ}$. As $E$ is $p$-horizontal, we know that $E^{\circ}$ is not $\mathcal{G}^{\circ}$-invariant (see \cite[\S\,3.5]{Druel2021}). Then \cite[Lemma~3.4 (2)]{Druel2021} implies that
		\[
		\mult_{E_1^{\circ}} (K_{\mathcal{G}_1^{\circ}} - (g^{\circ})^* K_{\mathcal{G}^{\circ}}) = r-1.
		\]
		Here notice that $(g^{\circ})^* K_{\mathcal{G}^{\circ}}$ is a well-defined Weil divisor because $g^{\circ}$ is finite.
 As $K_{\mathcal{G}_1^{\circ}}+\Delta_1^{\circ}=(g^{\circ})^*(K_{\mathcal{G}^{\circ}}+\Delta^{\circ})$, where $\Delta^{\circ}\coloneqq \Delta|_{U^{\circ}}$, one gets
		\begin{equation}
			\label{eq.DiscreFiniteCover}
			\mult_{E_1^{\circ}} (\Delta_1^{\circ}) + r - 1 = r\mult_{E^{\circ}}(\Delta^{\circ}) = r\mult_E(\Delta).
		\end{equation}
		By Lemma \ref{lem.Exc-div-family-leaves}, we have $\mult_{E_1^{\circ}}(\Delta_1^{\circ})\geq 1$ as $\Delta_1^{\circ}$ is integral effective, so one derives $\mult_E(\Delta)\geq 1$ from \eqref{eq.DiscreFiniteCover}.
	\end{proof}

\subsection{Kawamata--Miyaoka type inequality for $\epsilon$-lc Fano varieties}

We start with the following observation, which is a consequence of Proposition~\ref{prop.ExcDivFamLeaves}.

\begin{prop}
	\label{prop.SlopeRkOneSubsheaf}
    Let $0< \epsilon\leq 1$ be a real number. Let $X$ be a $\mathbb Q$-factorial $\epsilon$-lc Fano variety of dimension $n\geq 2$ and of Picard number $1$. Then for any rank $1$ subsheaf $\mathcal{L}$ of $\mathcal{T}_X$, the divisor class $c_1(X)-(1+\epsilon)c_1(\mathcal{L})$ is nef.
\end{prop}

\begin{proof}
	Without loss of generality, we may assume that $c_1(\mathcal{L})$ is ample and $\mathcal{T}_X/\mathcal{L}$ is torsion-free, so $\mathcal{L}$ is a foliation on $X$. Moreover, by a criterion of Bogomolov and McQuillan (see \cite[Theorem~1]{KebekusSolaCondeToma2007}), the foliation $\mathcal{L}$ is algebraically integrable such that the closure of its general leaf is a rational curve. In other words, a general fiber of $p$ in the family of leaves \eqref{eq.familyofleaves} of $\mathcal{L}$ is isomorphic to $\mathbb{P}^1$. Moreover, since $\rho(X)=1$ and $X$ is $\mathbb{Q}$-factorial, there always exist $p$-horizontal $e$-exceptional prime divisors, namely $E_1,\dots,E_{m}$.
	
	Let $F\cong \mathbb{P}^1$ be a general fiber of $p$. Then $U$ is smooth along $F$. Write $K_{e^{-1}\mathcal{L}}+\Delta \sim_{\mathbb{Q}} e^*K_{\mathcal{L}}$. Then $E_i\cdot F$ is a positive integer as $E_i$ is Cartier along $F$. By Proposition~\ref{prop.ExcDivFamLeaves}, we have
	\[
	0 < -e^*K_{\mathcal{L}}\cdot F = -(K_{e^{-1}\mathcal{L}}+\Delta)\cdot F= 2 - \sum_{i=1}^m \mult_{E_i}(\Delta) E_i\cdot F \leq 2-m.
	\]
	It follows that $m=1$, $E_1\cdot F=1$ and $-e^*K_{\mathcal{L}}\cdot F\leq 1$. Now write $K_{U}=e^*K_X+D$ for some $e$-exceptional $\mathbb{Q}$-divisor $D$. As $X$ has $\epsilon$-lc singularities, the coefficients of $D$ are at least $\epsilon-1$. As a consequence, we obtain
	\[
	-e^*K_X\cdot F = (-K_U + D)\cdot F = 2 +\mult_{E_1}(D) E_1\cdot F \geq 1 + \epsilon \geq (1+\epsilon)(-e^*K_{\mathcal{L}}\cdot F). 
	\]
	Now the result follows as $\rho(X)=1$.
\end{proof}

Now we are in the position to finish the proof of Theorem \ref{thm.KMinqEpsilonLC}.
	
	\begin{proof}[Proof of Theorem \ref{thm.KMinqEpsilonLC}]
		By the $\mathbb{Q}$-Bogomolov--Gieseker inequality and \cite{GrebKebekusPeternell2021}*{Theorem 1.2}, we have
 \[
 c_1(X)^{n}<\frac{2n}{n-1} \hat{c}_2(X)\cdot c_1(X)^{n-2}\leq \frac{2(1+\epsilon)}{\epsilon} \hat{c}_2(X)\cdot c_1(X)^{n-2}
 \]
 if $\mathcal{T}_X$ is semistable with respect to $c_1(X)$.
 So we may assume that $\mathcal{T}_X$ is not semistable with respect to $c_1(X)$. Applying Theorem~\ref{thm.LangerIneqII} to $\mathcal{T}_X$ and $\alpha_i=c_1(X)$ yields
		\begin{align*}
			\hat{\Delta}(\mathcal{T}_X)\cdot c_1(X)^{n-2} + n\left(\mu_{\max}(\mathcal{T}_X)-\mu(\mathcal{T}_X)\right)
			& \geq \frac{n^2\left(\mu_{\max}(\mathcal{T}_X)-\mu(\mathcal{T}_X)\right) \cdot \mu_{\min}(\mathcal{T}_X)}{n\mu(\mathcal{T}_X)}.
		\end{align*}
		By \cite[Proposition~3.6]{LiuLiu2023}, we have $\mu_{\min}(\mathcal{T}_X)>0$. As $\mu_{\max}(\mathcal{T}_X)>\mu(\mathcal{T}_X)$, the above inequality yields that
		\begin{align*}
			\left(2\hat{c}_2(X) - c_1(X)^2\right)\cdot c_1(X)^{n-2} + \mu_{\max}(\mathcal{T}_X) > 0.
		\end{align*}
		Denote by $\mathcal{F}$ the maximal destabilizing subsheaf of $\mathcal{T}_X$. Then \cite[Proposition~3.6]{LiuLiu2023} implies that $c_1(\mathcal{T}_X)-c_1(\mathcal{F})$ is ample. Combining this fact with Proposition~\ref{prop.SlopeRkOneSubsheaf} yields
		\[
		\mu_{\max}(\mathcal{T}_X) = \mu(\mathcal{F})=\frac{c_1(\mathcal{F})\cdot c_1(X)^{n-1}}{\rank (\mathcal{F})} \leq \begin{dcases}\frac{1}{1+\epsilon} c_1(X)^n &\text{if } \rank (\mathcal{F})=1;\\
		   \frac{1}{2} c_1(X)^n &\text{if }  \rank (\mathcal{F})\geq 2,
		\end{dcases}
		\]
		and the desired inequality follows immediately.
	\end{proof}
	
	\begin{ex}
		Let $X=\mathbb{P}(1,1,d)$ be the cone over a rational normal curve of degree $d$. Then $X$ has $\frac{2}{d}$-lc singularities and $\mathcal{O}_X(-K_X)\simeq \mathcal{O}_X(d+2)$. Consider the foliation $\mathcal{L}$ defined by the projection to the first two coordinates $X\dashrightarrow \mathbb{P}^1$. Then $\mathcal{L}\simeq \mathcal{O}_X(d)$. Moreover, we have
		\[
		\hat{c}_2(X)=(2d+1)c_1(\mathcal{O}_X(1))^2\quad \text{and} \quad c_1(X)^2 = (d+2)^2c_1(\mathcal{O}_X(1))^2.
		\]
		This yields
		\[
		{\hat{c}_2(X)} = \frac{2d+1}{(d+2)^2}{c_1(X)^2}.
		\]
		So Proposition~\ref{prop.SlopeRkOneSubsheaf} is sharp and $\epsilon$-lc singularities can not be replaced by klt singularities in Theorem~\ref{thm.KMinqEpsilonLC}.
	\end{ex}

\subsection{Kawamata--Miyaoka type inequality for canonical Fano $3$-folds}

Throughout this subsection, we assume that $X$ is a $\mathbb{Q}$-factorial canonical Fano $3$-fold of Picard number $1$. As canonical singularities are $1$-lc singularities, Theorem~\ref{thm.KMinqEpsilonLC} implies that 
\[c_1(X)^3 < 4\hat{c}_2(X) \cdot c_1(X).\]
We slightly improve it by looking deeper into the Harder--Narasimhan filtration of $\mathcal{T}_X$.
	
\begin{thm}\label{thm.kmineqfor3fold}
 Let $X$ be a $\mathbb{Q}$-factorial canonical Fano $3$-fold of Picard number $1$. Let $q:=\qQ(X)$ be the $\mathbb{Q}$-Fano index of $X$. Then
 \[
 c_1(X)^3 \leq 
 \begin{dcases}
 \frac{16}{5} \hat{c}_2(X)\cdot c_1(X) & \text{if } q\leq 5;\\
 \frac{4q^2}{q^2 + 2q -4} \hat{c}_2(X)\cdot c_1(X)& \text{if } q\geq 6.
 \end{dcases}
 \]
\end{thm}
	
	\begin{proof}
		By the $\mathbb{Q}$-Bogomolov--Gieseker inequality, we may assume that $\mathcal{T}_X$ is not semistable with respect to $c_1(X)$. Let $0=\mathcal{E}_0\subsetneq \mathcal{E}_1\subsetneq \dots \subsetneq \mathcal{E}_l=\mathcal{T}_X$ be its Harder--Narasimhan filtration. Then $2\leq l\leq 3$. Denote by $r_i$ the rank of $\mathcal{F}_i\coloneqq (\mathcal{E}_i/\mathcal{E}_{i-1})^{**}$ and by $q_i\geq 1$ the unique positive integer such that $c_1(\mathcal{F}_i)\equiv q_i A$, where $A$ is an ample generator of $\Cl (X)/\sim_{\mathbb{Q}}$ (see \cite[Proposition~3.6]{LiuLiu2023}). Then we have
		\[
		\sum_{i=1}^l r_i=3,\quad \sum_{i=1}^l q_i = q, \quad \text{and}\quad \frac{q_1}{r_1} > \dots > \frac{q_l}{r_l}>0.
		\]
		The proof will be divided into three different cases: $(l,r_1)=(2,1)$, $(2,2)$, or $(3,1)$.
		
		\begin{case}
			First we consider the case $(l,r_1)=(2,1)$. By Proposition~\ref{prop.SlopeRkOneSubsheaf}, we have $2q_1\leq q$. On the other hand, $q_1=\frac{q_1}{r_1}>\frac{q_2}{r_2}=\frac{q-q_1}{2}$ implies that $3q_1>q$. So Lemma~\ref{lem.LangerIneqI} implies that
			\[
			6\hat{c}_2(X)\cdot c_1(X) - 2 c_1(X)^3 
			\geq - \frac{(3q_1-q)^2}{2q^2} c_1(X)^3 \geq - \frac{1}{8} c_1(X)^3,
			\]
			which yields
			\[
			c_1(X)^3 \leq \frac{16}{5} \hat{c}_2(X)\cdot c_1(X).
			\] 
		\end{case}
		
		\begin{case}
			Next we consider the case $(l,r_1)=(2,2)$. Then \cite[Proposition~3.6]{LiuLiu2023} implies that $q_1<q$, and $\frac{q_1}{r_1}>\frac{q_2}{r_2}$ implies that $3q_1>2q$. By Lemma~\ref{lem.LangerIneqI} again, one gets
			\[
			6 \hat{c}_2(X)\cdot c_1(X) - 2c_1(X)^3
			\geq -\frac{(3q_1-2q)^2}{2q^2} c_1(X)^3 \geq -\frac{(q-3)^2}{2q^2} c_1(X)^3,
			\]
 where we used the fact that $q_1\leq q-1$. 
Then 
			\[
			c_1(X)^3 \leq \frac{4q^2}{q^2 + 2q- 3} \hat{c}_2(X)\cdot c_1(X).
			\]
		\end{case}
		
		\begin{case}
			Finally we consider the case $(l,r_1)=(3,1)$. As $q_1>q_2>q_3\geq 1$, we get $q\geq 6$ in this case. By Proposition~\ref{prop.SlopeRkOneSubsheaf}, we have $2q_1\leq q$. Lemma~\ref{lem.LangerIneqI} yields
			\begin{align*}
				6\hat{c}_2(X)\cdot c_1(X) - 2 c_1(X)^3 
				& \geq - \left((q_1-q_2)^2 + (2q_1+q_2-q)^2 + (q_1+2q_2-q)^2\right) A^2\cdot c_1(X) \\
				& \geq - \left(\left(\frac{q}{2} - q_2\right)^2 + q_2^2 + \left(2q_2-\frac{q}{2}\right)^2\right) A^2\cdot c_1(X) \\
				& = - \left(6 q_2^2 - 3q q_2 + \frac{q^2}{2}\right)\cdot \frac{1}{q^2} c_1(X)^3 \\
				& \geq - \frac{q^2-6q+12}{2q^2} c_1(X)^3,
			\end{align*}
			where the last inequality follows from $2\leq q_2\leq q_1-1\leq q/2-1$. So we obtain
			\[
			c_1(X)^3 \leq \frac{4q^2}{q^2 + 2q- 4} \hat{c}_2(X)\cdot c_1(X).
			\]
		\end{case}
		We finish the proof by comparing the inequalities in the three cases above. 
	\end{proof}

\section{Degrees, Fano indices and Chern classes}

\subsection{Connection between degrees and Fano indices}\label{sec.lowerbound}

In this subsection, we study the connection between degrees and Fano indices for canonical Fano varieties. 

\begin{lem}[{cf. \cite[Lemma~2.3]{Jiang2016}}]\label{lem DD is integer}
 Let $X$ be a normal projective variety of dimension $n\geq 2$.
 Let $D_1,\dots, D_{n-2}$ be Cartier divisors on $X$ and let $D, D'$ be Weil divisors on $X$ such that $D$ is Cartier in codimension $2$. Then
 \[
 (D_1\cdot D_2\cdots D_{n-2}\cdot D\cdot D')\in \mathbb{Z}. 
 \]
\end{lem}

\begin{proof}
If $n=2$, then clearly $(D\cdot D')\in \mathbb{Z}$ as $D$ is a Cartier divisor and $D'$ is a Weil divisor.

In general, we can express $D_1=H_1-H_2$ where $H_1$ and $H_2$ are general very ample divisors on $X$.
Then by induction on $n$, for $i=1,2$,
\[(H_i\cdot D_2\cdots D_{n-2}\cdot D\cdot D')=(D_2|_{H_i}\cdots D_{n-2}|_{H_i}\cdot D|_{H_i}\cdot D'|_{H_i})\in \mathbb{Z}. \]
This proves the assertion.
\end{proof}

\begin{thm}\label{thm.key}
 Let $r$ and $q$ be positive integers. Let $X$ be a canonical Fano variety of dimension $n\geq 2$ such that $-rK_X$ is Cartier. 
 Let $A$ be a $\mathbb{Q}$-Cartier Weil divisor such that $-K_X\equiv qA$. Take $J_A$ to be the smallest positive integer such that $J_AA$ is Cartier in codimension $2$. 
 Then 
 \begin{enumerate}
 \item $ J_Ar^{n-2}(-K_X)^n/q^2$ is a positive integer; 

 \item if $-K_X\sim qA$, then $J_A\mid q$ and $q \mid r^{n-2}(-K_X)^n$.
 \end{enumerate}
\end{thm}

\begin{proof}
By Lemma~\ref{lem DD is integer}, \[\frac{J_Ar^{n-2}}{q^2}(-K_X)^n=(-rK_X)^{n-2}\cdot J_AA\cdot A\] is an integer, and it is positive as $-K_X$ is ample. This proves the first assertion.

For the second assertion, just notice that $qA\sim -K_X$ is Cartier in codimension $2$ by \cite[{Corollary~5.18}]{KollarMori1998}, hence $J_A\mid q$ by the definition of $J_A$. 
\end{proof}


In practice, we just take $r=r_X$ and $q=\qW(X)$ in Theorem~\ref{thm.key}.

\begin{ex}[{\cite[{Table~1}]{DedieuSernesi2023}}]\label{ex.gor}
$X=\mathbb P(1,6,14,21)$ is a Gorenstein weighted projective $3$-fold. Then $-K_X\sim 42A$ where $A=\mathcal{O}_X(1)$ and $J_A=\qW(X)=r_X(-K_X)^3=42$.
\end{ex}

\subsection{Difference between Chern classes and generalized Chern classes}

In this subsection, we study the difference between Chern classes and generalized Chern classes for canonical Fano $3$-folds. 
 
\begin{defn}\label{def ecgc}
Let $X$ be a normal projective $3$-fold with canonical singularities.
For an irreducible curve $C\subset \textrm{Sing}(X)$, we say that $C\subset X$ is {\it of type $\mathsf{T}$} if at a general point of $C$, $X$ is analytically isomophic to $\mathbb{A}^1\times S_C$ where $S_C$ is a Du Val singularity of type $\mathsf{T}$, or equivalently, for a general hyperplane $H$ on $X$, $H$ has Du Val singularities of type $\mathsf{T}$ in a neighborhood of any point of $H\cap C$. Here $\mathsf{T}\in \{\mathsf{A}_n, \mathsf{D}_m, \mathsf{E}_k\mid n\geq 1, m\geq 4, k=6,7,8\}.$ 

We define $e_C$ to be $1$ plus the number of exceptional curves on the minimal resolution of $S_C$, and define $g_C$ to be 
 the order of the local fundamental group of $S_C$, and define $j_C$ to be 
 the order of the Weil divisor class group of $S_C$ (see \cite[{Remark~4.2.9}]{Kawakita2024}). 
 Namely, 
\begin{align}(e_C, g_C, j_C)=\begin{cases} (n+1, n+1, n+1) & \text{ if } C\subset X \text{ is of type } \mathsf{A}_n;\\
(m+1, 4m-8, 4) & \text{ if } C\subset X \text{ is of type } \mathsf{D}_m;\\
(7, 24, 3) & \text{ if } C\subset X \text{ is of type } \mathsf{E}_6;\\
(8, 48, 2) & \text{ if } C\subset X \text{ is of type } \mathsf{E}_7;\\
(9, 120,1) & \text{ if } C\subset X \text{ is of type } \mathsf{E}_8.
 \end{cases}\label{eq ecgc}
\end{align}
\end{defn}

We have the following easy inequality. 
\begin{lem}\label{lem j<e-1/g}
 In \eqref{eq ecgc}, we have $j_C-\frac{1}{j_C}\leq e_C-\frac{1}{g_C}.$
\end{lem}

\begin{thm}\label{thm.diffcandgc}
 Let $X$ be a normal projective $3$-fold with canonical singularities. 
 Then for any $\mathbb{Q}$-Cartier $\mathbb{Q}$-divisor $H$ on $X$, we have
 \[
 c_2(X)\cdot H-\hat c_2(X) \cdot H=\sum_{C\subset \text{\rm Sing}(X)}\left(e_C-\frac{1}{g_C}\right)(H\cdot C),
 \]
 where the sum runs over irreducible curves $C\subset \text{\rm Sing}(X)$ and $e_C, g_C$ are defined in Definition~\ref{def ecgc}.
\end{thm}

\begin{proof}
Note that the desired equality is linear in $H$. 
As any $\mathbb{Q}$-Cartier $\mathbb{Q}$-divisor can be written as a linear combination of very ample divisors, 
we may assume that $H$ is very ample and sufficiently general. By Bertini's theorem, $H$ has canonical singularities. 

 Let $f\colon Y\to X$ be a resolution of singularities such that $f|_{H'}\colon H'\to H$ is the minimal resolution of $H$, where $H'=f^*H=f^{-1}(H)$.
 Here we explain that such resolution exists: 
first take a terminalization $f_1\colon Y_1\to X$ where $Y_1$ has terminal singularities and $f_1^*K_X=K_{Y_1}$ (see \cite[Theorem~6.23]{KollarMori1998}), then take $f_2\colon Y\to Y_1$ to be a resolution of singularities which is isomorphic outside $\textrm{Sing}(Y_1)$. Recall that $\textrm{Sing}(Y_1)$ is a finite set of points. Then as $H$ is a general hyperplane section, $f_1$ induces the minimal resolution $H'\to H$ and $f_2$ induces an isomorphism over $H'$.

 Then $c_2(X)\cdot H=c_2(Y)\cdot H'=c_2(\mathcal T_Y|_{H'})$.
 From the exact sequence
 \[
 0\to \mathcal T_{H'}\to \mathcal T_Y|_{H'}\to \mathcal O_{H'}(H')\to 0,
 \]
 we have 
 \begin{align}
 c_2(\mathcal T_Y|_{H'})=c_2(\mathcal T_{H'})+c_1(\mathcal T_{H'})\cdot H'|_{H'}. \label{eq c21}
 \end{align} 

According to \cite[Proposition~3.11]{GrebKebekusPeternellTaji2019a}, the surface $H$ admits a structure of a $\mathbb{Q}$-variety, compatible with that over a big open subset of $X$, such that the following sequence of $\mathbb{Q}$-vector bundles 
 \[
 0\to \mathcal T_H\to \mathcal T_X|_H\to \mathcal O_H(H)\to 0
 \]
 is $\mathbb{Q}$-exact (cf. \cite{Kawamata1992a}*{\S 2} and \cite[Construction 3.8]{GrebKebekusPeternellTaji2019a}) and $\hat c_2(X)\cdot H=\hat c_2(\mathcal T_X|_H)$. By \cite{Kawamata1992a}*{Lemmas 2.1 and 2.2}, we have 
\begin{align}
 \hat c_2(\mathcal T_X|_H)=\hat c_2(\mathcal T_H)+\hat c_2(\mathcal O_H(H)) + c_1(\mathcal T_{H})\cdot H|_H.
\label{eq c22}
 \end{align} 
 Notice that $\hat c_2(\mathcal O_H(H))=0$ as $\mathcal{O}_H(H)$ is invertible and $c_1(\mathcal T_{H'})\cdot H'|_{H'}=c_1(\mathcal T_{H})\cdot H|_H$ by the projection formula. Hence by \eqref{eq c21} and \eqref{eq c22}, we have
 \begin{align}\label{eq c2-c2}
 c_2(X)\cdot H-\hat c_2(X)\cdot H=c_2(\mathcal T_{H'})-\hat c_2(\mathcal T_H).
 \end{align}
 By \cite[{Definition~10.7, Theorem~10.8}]{1992}, we have 
 \begin{align}\label{eq c2=eorb}
 \hat c_2(\mathcal T_H)=e_{orb}(H)=e_{top}(H)-\sum_{C\subset \text{\rm Sing}(X)}\left(1-\frac{1}{g_C}\right)(C \cdot H).
 \end{align}
 For an irreducible curve ${C\subset \text{\rm Sing}(X)}$, over any point in $C\cap H$, the exceptional set of $f|_{H'}\colon H'\to H$ is a tree of $e_C-1$ rational curves, whose topological Euler number is $e_C$. So we get
 \begin{align}\label{eq etop-etop}
e_{top}(H')-e_{top}(H)=\sum_{C\subset \text{\rm Sing}(X)}(e_C-1)(C \cdot H).
 \end{align}
 As $c_2(\mathcal T_{H'})=e_{top}(H')$, the conclusion follows from combining \eqref{eq c2-c2}, \eqref{eq c2=eorb}, and \eqref{eq etop-etop}.
\end{proof}

Applying Theorem~\ref{thm.diffcandgc} to $H=-r_XK_X$, we get the following by Theorem~\ref{thm.kmineqfor3fold}: 

\begin{cor}\label{cor.rangeofJ}
 Let $X$ be a $\mathbb Q$-factorial canonical Fano $3$-fold of Picard number $1$. Let $A$ be an ample Weil divisor generating $\Cl (X)/\sim_{\mathbb{Q}}$. Take $J_A$ to be the smallest positive integer such that $J_AA$ is Cartier in codimension $2$. Let $J_A=p_1^{a_1}p_2^{a_2}\cdots p_k^{a_k}$ be the prime factorization, where $p_i$ are distinct prime numbers. Then
 \begin{align}
 \sum_{i=1}^k\left(p_i^{a_i}-\frac{1}{p_i^{a_i}}\right)< r_Xc_2(X) \cdot c_1(X)- \frac{r_X}{4}c_1(X)^3.\label{eq upper J1}
 \end{align}
 Moreover, if $q:=\qQ(X)\geq 6$, then
 \begin{align}\sum_{i=1}^k\left(p_i^{a_i}-\frac{1}{p_i^{a_i}}\right)\leq r_Xc_2(X) \cdot c_1(X)- \frac{r_X(q^2+2q-4)}{4q^2}c_1(X)^3.\label{eq upper J2}
 \end{align}
\end{cor}

\begin{proof}
Applying Theorem~\ref{thm.diffcandgc} to $H=-r_XK_X$, as $(-r_XK_X)\cdot C\geq 1$ for any curve $C$ on $X$, 
we have 
\begin{align}
 \sum_{C\subset \text{\rm Sing}(X)}\left(e_C-\frac{1}{g_C}\right) \leq 
 r_Xc_2(X) \cdot c_1(X)-r_X\hat c_2(X) \cdot c_1(X). \label{eq J1.1}
\end{align}
 By Lemma~\ref{lem j<e-1/g}, we have 
\begin{align} \sum_{C\subset \text{\rm Sing}(X)}\left(e_C-\frac{1}{g_C}\right)\geq \sum_{C\subset \text{\rm Sing}(X)}\left(j_C-\frac{1}{j_C}\right).
\label{eq J1.2}
\end{align}
By the definitions of $J_A$ and $j_C$, we know that $J_A$ divides $\lcm\{j_C\mid C\subset \text{\rm Sing}(X)\}$, so each $p_i^{a_i}$ divides at least one $j_C$ for some $C\subset \text{\rm Sing}(X)$. So by \cite[Page~65, (2.2)]{ChenJiang2016}, 
\begin{align}
 \sum_{C\subset \text{\rm Sing}(X)}\left(j_C-\frac{1}{j_C}\right)\geq \sum_{i=1}^k\left(p_i^{a_i}-\frac{1}{p_i^{a_i}}\right).\label{eq J1.3}
\end{align}
So we get the desired inequalities by combining \eqref{eq J1.1}, \eqref{eq J1.2}, \eqref{eq J1.3}, and Theorem~\ref{thm.kmineqfor3fold}. 
\end{proof}

\begin{lem}\label{lem J<9.5}
 Let $J=p_1^{a_1}p_2^{a_2}\cdots p_k^{a_k}$ be the prime factorization of a positive integer $J$, where $p_i$ are distinct prime numbers.
 If $ \sum_{i=1}^k(p_i^{a_i}-\frac{1}{p_i^{a_i}})<9.5$,
 then
 \[J\in \{1,2,3,4,5,6,7,8,9,10,12,14,15,20,30\}.\]
\end{lem}

\begin{proof}
 It is clear that $p_i^{a_i}\in \{2,3,4,5,7,8,9\}$. So all possibilities of $J$ can be obtained by a direct computation.
\end{proof}

 \section{Upper bound of degrees}

In this section, we study the upper bound for degrees of $\mathbb Q$-factorial canonical Fano $3$-folds of Picard number $1$. 

First we give a reduction which works for all canonical weak Fano $3$-folds with large $\qQ(X)$.
\begin{lem}\label{lem modification}
 Let $X$ be a canonical weak Fano $3$-fold with $\qQ(X)\geq 7$. 
 Suppose that $\Cl (X)$ has an $s$-torsion element, where $s\geq 1$ is a positive integer.
 Then there exsits a $3$-fold $X'$ with the following properties:
 \begin{enumerate}
 \item $X'$ is a $\mathbb Q$-factorial canonical Fano $3$-fold of Picard number $1$;
 
 \item $\qQ(X')\geq \qQ(X)\geq 7$;

 \item $(-K_{X'})^3\geq s(-K_X)^3$. 
 \end{enumerate}
\end{lem}
\begin{proof}
Set $q\coloneq \qQ(X)$. Then we have $-K_X\sim_{\mathbb{Q}}qA$ where $A$ is a nef and big $\mathbb{Q}$-Cartier Weil divisor.

The $s$-torsion Weil divisor in $\Cl (X)$ induces a quasi-\'etale morphism $\pi\colon W\to X$ with $\deg \pi=s$ (\cite[Definition~2.52]{KollarMori1998}). In particular, $\pi^*K_X=K_W$. Hence $W$ is a canonical weak Fano $3$-fold by \cite[Proposition~5.20]{KollarMori1998} and $(-K_{W})^3=s(-K_{X })^3$. 

Let $\phi\colon W'\to W$ be a $\mathbb{Q}$-factorialization such that $W'$ is a $\mathbb{Q}$-factorial canonical weak Fano $3$-fold with $-K_{W'}\sim_{\mathbb{Q}}q \phi^* \pi^*A$ where $\phi^* \pi^*A$ is a Weil divisor as $\pi$ is finite and $\phi$ is small. We can run a $K$-MMP on $W'$ which ends up with a Mori fiber space $X'\to T$ where $X'$ is $\mathbb{Q}$-factorial and canonical. Then $-K_{X'}\sim_{\mathbb{Q}}q A'$ where $A'$ is the strict transform of $\phi^* \pi^*A$ on $X'$.

If $\dim T\geq 1$, then for a general fiber $F$ of $X'\to T$, we have $-K_F\sim_{\mathbb{Q}}q A'|_F$ and $F$ is either $\mathbb{P}^1$ or a canonical del Pezzo surface, but this contradicts \cite[Proposition~3.3]{WangCX2024} as $q\geq 7$.

Hence $\dim T=0$, which means that $X'$ is a $\mathbb Q$-factorial canonical Fano $3$-fold of Picard number $1$. From the construction, we get $ \qQ(X')\geq q$, and 
\[
 (-K_{X'})^3\geq (-K_{W'})^3=(-K_{W})^3=s(-K_{X })^3,
\] 
where the first inequality is by \cite[Lemma~4.4]{Jiang2021}.
\end{proof}

\begin{prop}\label{prop modification}
 Let $X$ be a canonical weak Fano $3$-fold with $\qQ(X)\geq 7$.
 Then there exists a $3$-fold $Y$ with the following properties:
 \begin{enumerate}
 \item $Y$ is a $\mathbb Q$-factorial canonical Fano $3$-fold of Picard number $1$;

 \item $\Cl (Y)$ is torsion-free;
 
 \item $\qW(Y)=\qQ(Y)\geq \qQ(X)\geq 7$;

 \item $(-K_Y)^3\geq (-K_X)^3$. 
 \end{enumerate}
\end{prop}

\begin{proof}
We can apply Lemma~\ref{lem modification} repeatedly to get such $Y$ with $\Cl (Y)$ torsion-free. Here this process stops as $(-K_{X'})^3\leq 324$ for a canonical Fano $3$-fold $X'$ by \cite[{Theorem~1.1}]{JiangZou2023}.
\end{proof}

\begin{lem}\label{lem.upbisoforq<=6}
 Let $X$ be a $\mathbb Q$-factorial canonical Fano $3$-fold of Picard number $1$ with $\qQ(X)\leq 6$. Then $(-K_X)^3\leq 72$, and the equality holds if and only if $X\cong \mathbb P(1,1,1,3)$. 
\end{lem}

\begin{proof}
 If $X$ is Gorenstein, then by \cite[Theorem~1.5]{Prokhorov2005}, we have $(-K_X)^3\leq 72$ and the equality holds only if $X\cong \mathbb P(1,1,1,3)$ or $\mathbb P(1,1,4,6)$, where $\qQ(X)=6$ or $12$ respectively. 

 From now on, assume that $(-K_X)^3\geq 72$ and $X$ is not Gorenstein, that is, $\mathcal{R}_X\neq \emptyset$.
 
 If $\qQ(X)\leq 5$, then by Theorems~\ref{thm.kmineqfor3fold} and~\ref{thm.diffcandgc}, we have
 \[
 22.5\leq \frac{5}{16} c_1(X)^3\leq c_2(X)\cdot c_1(X)\stackrel{\eqref{eq.range}}{=} 24-\sum_{r_i\in \mathcal{R}_X}(r_i-\frac{1}{r_i}) \leq 24-\frac{3}{2}=22.5. 
 \]
Then $c_1(X)^3=72$, $\mathcal{R}_X=\{2\}$, and $B_X=\{(2,1)\}$. But then $h^0(X,-K_X)=36+3-1/4\in \mathbb Z$ by \eqref{eq.RR-Fano}, which is absurd.

 If $\qQ(X)=6$, then by Theorems~\ref{thm.kmineqfor3fold} and~\ref{thm.diffcandgc}, we have
 \[
 22\leq \frac{11}{36}c_1(X)^3\leq c_2(X)\cdot c_1(X) \stackrel{\eqref{eq.range}}{=}24-\sum_{r_i\in \mathcal{R}_X}(r_i-\frac{1}{r_i})\leq 22.5.
 \]
 It follows that $\mathcal{R}_X=\{2\}$. Then the above inequality and 
 \eqref{eq.RR-Fano} implies that $r_Xc_1(X)^3=145.$
 By Theorem~\ref{thm.key}, 
 $\qW(X)\mid r_Xc_1(X)^3$ while $\qW(X)\mid \qQ(X)=6$, so $\qW(X)=1$. Hence there exists a torsion element in $\Cl (X)$ with order $6$ by Lemma~\ref{lem ex torsion}. Then by the index one cover,
 there is a canonical Fano $3$-fold $X'$ with $(-K_{X'})^3\geq 6(-K_X)^3>324$, contradicting \cite[{Theorem~1.1}]{JiangZou2023}.
\end{proof}

\begin{thm}\label{thm.sharpboundofdegree}
 Let $X$ be a $\mathbb Q$-factorial canonical Fano $3$-fold of Picard number $1$. Then $(-K_X)^3\leq 72$ and the equality holds if and only if $X\cong \mathbb P(1,1,1,3)$ or $\mathbb P(1,1,4,6)$.
\end{thm}

\begin{proof}We may assume that $(-K_X)^3\geq 72$.
By Lemma~\ref{lem.upbisoforq<=6}, we may assume that $\qQ(X)\geq 7$. By applying Lemma~\ref{lem modification} repeatedly, we may assume that $\Cl (X)$ is torsion-free and $q:=\qW(X)=\qQ(X)\geq 7$. 
 By \cite[Theorem~1.5]{Prokhorov2005}, we may assume that $X$ is not Gorenstein. 
 Then Theorems~\ref{thm.kmineqfor3fold} and~\ref{thm.diffcandgc} imply that
 \[
 18\leq \frac{1}{4}c_1(X)^3< c_2(X) \cdot c_1(X)\stackrel{\eqref{eq.range}}{=} 24-\sum_{r_i\in \mathcal{R}_X}(r_i-\frac{1}{r_i}). 
 \]
 That is, $\sum_{r_i\in \mathcal{R}_X}(r_i-\frac{1}{r_i})<6$. 
 So there are in total 11 possibilities for $\mathcal R_X$: 
 \[
 \mathcal R_X \in \{ \{2\}, \{3\},\{4\},\{5\},\{6\},\{2,2\},\{2,3\},\{2,4\},\{3,3\},\{2,2,2\},\{2,2,3\}\}.
 \]
 For each $\mathcal{R}_X$, the corresponding $B_X$ consists of some of $(2,1)$, $(3,1)$, $(4,1)$, $(5,1)$, $(5,2)$, $(6,1)$.
So we can list all possibilities of $r_Xc_1^3$ satisfying \eqref{eq.RR-Fano} in Table~\ref{tab1}.

 {
\begin{longtable}{LLLLL}
\caption{}\label{tab1}\\
\hline
 \mathcal{R}_X & r_X& r_Xc_2c_1 & r_Xc_1^3 & r_Xc_2c_1-r_Xc_1^3/4\\
\hline
\endfirsthead
\multicolumn{4}{l}{{ {\bf \tablename\ \thetable{}} \textrm{-- continued}}}
\\
\hline 
 \mathcal{R}_X & r_X& r_Xc_2c_1 & r_Xc_1^3& r_Xc_2c_1-r_Xc_1^3/4\\
\endhead
\hline
\hline \multicolumn{4}{c}{{\textrm{Continued on next page}}} \\ \hline
\endfoot

\hline \hline
\endlastfoot
 \{2\} & 2& 45 & 145, 149,153,157,161,165,169,173,177 & \leq 8.75\\
 \{3\}& 3& 64 & 218, 224, 230, 236, 242, 248, 254 & \leq 9.5\\
 \{4\}& 4& 81 & 291, 299, 307, 315, 323 & \leq 8.25\\
 \{5\}& 5& 96 &364, 366, 374, 376 & \leq 5\\
 \{6\}& 6& 109 & \text{NO}\\
 \{2,2\}& 2& 42 & 146, 150, 154, 158, 162, 166& \leq 5.5\\
 \{2,3\}& 6& 119 & 439, 451, 463, 475& \leq 9.25\\
 \{2,4\}& 4& 75 & 293& 1.75\\
 \{3,3\}& 3& 56&220 & 1\\
 \{2,2,2\}&2& 39 & 147, 151, 155& \leq 2.25\\
 \{2,2,3\}&6& 110 & \text{NO} 
\end{longtable}
} 

Let $A$ be an ample Weil divisor generating $\Cl (X)$. Take $J_A$ to be the smallest positive integer such that $J_AA$ is Cartier in codimension $2$. Let $J_A=p_1^{a_1}p_2^{a_2}\cdots p_k^{a_k}$ be the prime factorization, where $p_i$ are distinct prime numbers.
Then by the last column of Table~\ref{tab1} and Corollary~\ref{cor.rangeofJ}, $J_A$ sastifies Lemma~\ref{lem J<9.5}.

If $J_A=q$, then we have 
\begin{itemize}
 \item $J_A=q\geq 7$ is a factor of $r_Xc_1(X)^3$ by Theorem~\ref{thm.key};
 \item $J_A\in \{7,8,9,10,12,14,15,20,30\}$ 
 by Lemma~\ref{lem J<9.5}, which implies that $\sum_{i=1}^k(p_i^{a_i}-\frac{1}{p_i^{a_i}})>6$;
 \item $r_Xc_2(X) \cdot c_1(X)- \frac{r_X}{4}c_1(X)^3>6$ by Corollary~\ref{cor.rangeofJ}.
\end{itemize}
All possibilities in Table~\ref{tab1}
satisfying these three properties
are picked out in Table~\ref{tab2}. But all of them contradict \eqref{eq upper J2}.

 {
\begin{longtable}{LLLLL}
\caption{}\label{tab2}\\
\hline
 \mathcal{R}_X & r_X& r_Xc_2c_1 & r_Xc_1^3 & q\\
\hline
\endfirsthead
\multicolumn{4}{l}{{ {\bf \tablename\ \thetable{}} \textrm{-- continued}}}
\\
\hline 
 \mathcal{R}_X & r_X& r_Xc_2c_1 & r_Xc_1^3 &  q\\
 \hline 
\endhead
\hline
\hline \multicolumn{4}{c}{{\textrm{Continued on next page}}} \\ \hline
\endfoot

\hline \hline
\endlastfoot
 \{2\} & 2& 45 & 153 & 9 \\
 \{3\}& 3& 64 & 224 & 7, 8, 14\\
  \{3\}& 3& 64 &  230 & 10  
\end{longtable}
} 

If $J_A\neq q$, then by Theorem~\ref{thm.key},
\begin{itemize}
 \item $r_Xc_1(X)^3$ contains a square factor $(\frac{q}{J_A})^2$;
 \item $J_A$ is a factor of $r_Xc_1(X)^3/(\frac{q}{J_A})^2$ satisfying Lemma~\ref{lem J<9.5}.
\end{itemize}
 All possibilities in Table~\ref{tab1} with these two properties are listed in Table~\ref{tab3}, where we cross-out $J_A$ which does not satisfy 
\eqref{eq upper J1}.
Then we can pick out those possibilities with $q\geq 7$ in Table~\ref{tab4}, but all of them contradict \eqref{eq upper J2}.
 
{
\begin{longtable}{LLLLLLL}
 \caption{}\label{tab3}\\
 \hline
 \mathcal{R}_X & r_X& r_Xc_2c_1 & r_Xc_1^3 & q/J_A& J_A & r_Xc_2c_1-r_Xc_1^3/4\\
 \hline
 \endfirsthead
 \multicolumn{4}{l}{{ {\bf \tablename\ \thetable{}} \textrm{-- continued}}}
 \\
 \hline 
 \mathcal{R}_X & r_X& r_Xc_2c_1 & r_Xc_1^3 & q/J_A& J_A & r_Xc_2c_1-r_Xc_1^3/4\\
 \endhead
 \hline
 \hline \multicolumn{4}{c}{{\textrm{Continued on next page}}} \\ \hline
 \endfoot
 
 \hline \hline
 \endlastfoot
 \{2\} & 2& 45 & 153=3^2\cdot 17 & 3& 1 \\
 \{2\} & 2& 45 & 169=13^2 & 13& 1 \\
 \{3\}& 3& 64 & 224=2^5\cdot 7 & 2& 1, 2, 4, 7, 8, \cancel{14} & 8\\
 \{3\}& 3& 64 & 224=2^5\cdot 7 & 4& 1, 2, 7, \cancel{14}& 8\\
 \{3\}& 3& 64 & 236=2^2 \cdot 59 &2 & 1\\ \{3\}& 3& 64 & 242=2\cdot 11^2 &11& 1,2\\ 
 \{3\}& 3& 64 & 248=2^3\cdot 31 &2 & 1,2\\
 \{4\}& 4& 81 & 315=3^2\cdot 5\cdot 7 &3 & 1,\cancel{5},\cancel{7} & 2.25\\
 \{5\}& 5& 96 &364=2^2\cdot 7\cdot 13 &2& 1,\cancel{7} & 5\\
 \{5\}& 5& 96 & 376=2^3\cdot 47 &2& 1,2 \\
 \{2,2\}& 2& 42 & 150=2\cdot 3\cdot 5^2 &5 & 1,2,3,6 \\
 \{2,2\}& 2& 42 & 162=2\cdot 3^4 & 3& 1,\cancel{2},\cancel{3},\cancel{6},\cancel{9} & 1.5\\
 \{2,2\}& 2& 42 & 162=2\cdot 3^4 & 9& 1,\cancel{2}& 1.5\\
 \{2,3\}& 6& 119 & 475=5^2\cdot 19 & 5& 1\\
 \{3,3\}& 3& 56&220=2^2\cdot 5\cdot 11 & 2& 1,\cancel{5} & 1\\
 \{2,2,2\}&2& 39 & 147=3\cdot 7^2 &7& 1,\cancel{3} & 2.25
\end{longtable}
} 
 
{
\begin{longtable}{LLLLL}
 \caption{}\label{tab4}\\
 \hline
 \mathcal{R}_X & r_X& r_Xc_2c_1 & r_Xc_1^3 & (q, J_A) \\
 \hline
 \endfirsthead
 \multicolumn{4}{l}{{ {\bf \tablename\ \thetable{}} \textrm{-- continued}}}
 \\
 \hline 
 \mathcal{R}_X & r_X& r_Xc_2c_1 & r_Xc_1^3 & (q, J_A) \\
 \endhead
 \hline
 \hline \multicolumn{4}{c}{{\textrm{Continued on next page}}} \\ \hline
 \endfoot
 
 \hline \hline
 \endlastfoot
 \{2\} & 2& 45 & 169=13^2 & (13, 1) \\
 \{3\}& 3& 64 & 224=2^5\cdot 7 & (8,4), (14, 7), (16, 8)\\
 \{3\}& 3& 64 & 224=2^5\cdot 7 & (8, 2), (28, 7)\\
 \{3\}& 3& 64 & 242=2\cdot 11^2 &(11, 1), (22,2)\\ 
 \{2,2\}& 2& 42 & 150=2\cdot 3\cdot 5^2 &(10, 2), (15, 3), (30, 6) \\
 \{2,2\}& 2& 42 & 162=2\cdot 3^4 & (9, 1)\\
 \{2,2,2\}&2& 39 & 147=3\cdot 7^2 &(7, 1)
\end{longtable}
} 
\end{proof}

\begin{cor}
 Let $X$ be a canonical weak Fano $3$-fold with $\qQ(X)\geq 7$. Then $(-K_X)^3\leq 72$.
\end{cor}

 \begin{proof}
 This directly follows from Proposition~\ref{prop modification} and Theorem~\ref{thm.sharpboundofdegree}.
 \end{proof}
 
As a corollary, we can solve the remaining case in \cite[Main~Theorem]{Lai2021} for $\mathbb{Q}$-factorial terminal weak Fano $3$-folds of Picard number $2$.

\begin{cor}[{cf. \cite[Main~Theorem]{Lai2021}}]\label{cor.pic2}
 Let $X$ be a $\mathbb{Q}$-factorial terminal weak Fano $3$-fold of Picard number $2$. Then $(-K_X)^3\leq 72$.
\end{cor}

\begin{proof}
 In \cite[Main~Theorem]{Lai2021}, the remaining case is when $X$ admits a $K$-trivial extremal contraction $X\to Y$. In this case, $Y$ is a $\mathbb Q$-factorial canonical Fano $3$-fold of Picard number $1$ and $(-K_Y)^3=(-K_X)^3$, so it follows from Theorem~\ref{thm.sharpboundofdegree}.
\end{proof}

Right after this paper putting on arXiv, the authors are informed kindly by  Ching-Jui Lai that the remaining case in \cite[Main~Theorem]{Lai2021} has been settled down by him and Lee in \cite{LaiLee2025}. Their approach develops the two rays method in \cite{Lai2021}, which is different from ours in Corollary \ref{cor.pic2}.

\begin{rem}
In the proof of Theorem~\ref{thm.sharpboundofdegree}, we use \cite[Theorem~1.5]{Prokhorov2005} to treat the Gorenstein case. However, if we just want to prove the upper bound $(-K_X)^3\leq 72$ without the characterization of the equality case, we can avoid the use of \cite[Theorem~1.5]{Prokhorov2005} by the following modification in the proofs of Lemma~\ref{lem.upbisoforq<=6} and Theorem~\ref{thm.sharpboundofdegree}.

From now on, suppose that $X$ is Gorenstein and $c_1(X)^3>72$. Take $q=\qQ(X)$.  Let $A$ be a Weil divisor such that $-K_X\equiv qA$. 
Recall that if $X$ is Gorenstein, then $c_2(X)\cdot c_1(X)=24$ and $c_1(X)^3$ is an even integer by Theorem~\ref{thm.euler}. In particular, $c_1(X)^3\geq 74.$

(1) In the proof of Lemma~\ref{lem.upbisoforq<=6}, we need to further treat the case that $q\leq 6$ and $X$ is Gorenstein. 

If $\Sing(X)$ contains a curve, then Theorems~\ref{thm.kmineqfor3fold} and~\ref{thm.diffcandgc} imply that
\[
 \frac{11}{36}c_1(X)^3\leq c_2(X)\cdot c_1(X)-\sum_{C\subset \text{\rm Sing}(X)}\left(e_C-\frac{1}{g_C}\right)(c_1(X)\cdot C)  \leq 24-(2-\frac{1}{2})= 22.5.
 \]
 In particular, $c_1(X)^3<74$. This is absurd.

If $\Sing(X)$ contains no curve, 
then the above inequality gives 
\[
 \frac{11}{36}c_1(X)^3\leq c_2(X)\cdot c_1(X)=24,
 \]
 which implies that $c_1(X)^3\in\{74,76, 78\}$. 
 On the other hand, by Lemma~\ref{lem A2K even}, $\frac{1}{q^2}c_1(X)^3=(A^2\cdot K_X)$ is an even integer. This means that $q=1$ and then \cite[Proposition 3.6]{LiuLiu2023} says that $-c_1(\mathcal{F})$ is nef for any proper subsheaf $\mathcal{F}$ of $\mathcal{T}_X$. So $\mathcal{T}_X$ is stable and the $\mathbb{Q}$-Bogomolov--Gieseker inequality implies 
\[
c_1(X)^3\leq 3\hat{c}_2(X)\cdot c_1(X) = 3c_2(X)\cdot c_1(X) = 72,
\]
which is a contradiction.

 (2) In the proof of Theorem~\ref{thm.sharpboundofdegree}, we need to further treat the case that $q=\qW(X)=\qQ(X)\geq 7$ and $X$ is Gorenstein. 

 Let $A$ be an ample Weil divisor generating $\Cl (X)$. Take $J_A$ to be the smallest positive integer such that $J_AA$ is Cartier in codimension $2$. Let $J_A=p_1^{a_1}p_2^{a_2}\cdots p_k^{a_k}$ be the prime factorization, where $p_i$ are distinct prime numbers.

Then by  Corollary~\ref{cor.rangeofJ},
\[\sum_{i=1}^k\left(p_i^{a_i}-\frac{1}{p_i^{a_i}}\right)< c_2(X) \cdot c_1(X)- \frac{1}{4}c_1(X)^3\leq 5.5.\]
This implies that $J_A\leq 6$ by direct computation. Also this inequality gives $c_1(X)^3<96$. So $c_1(X)^3\in \{74, 76, \dots, 94\}.$
As $J_A\neq q$, by Theorem~\ref{thm.key},
\begin{itemize}
 \item $c_1(X)^3$ contains a square factor $(\frac{q}{J_A})^2$;
 \item $J_A\leq 6$ is a factor of $c_1(X)^3/(\frac{q}{J_A})^2$.
\end{itemize}
 All possibilities with these two properties are listed in Table~\ref{tab5}, where we cross-out $J_A$ which does not satisfy 
\eqref{eq upper J1}. Hence the only possible case with $q\geq 7$ is $(c_1^3, q, J_A)=(80, 8, 4)$,  but this  contradicts \eqref{eq upper J2}.

{
\begin{longtable}{LLLL}
 \caption{}\label{tab5}\\
 \hline
 c_1^3 & q/J_A& J_A & c_2c_1-c_1^3/4\\
 \hline
 \endfirsthead
 \multicolumn{4}{l}{{ {\bf \tablename\ \thetable{}} \textrm{-- continued}}}
 \\
 \hline 
  c_1^3 & q/J_A& J_A & c_2c_1-c_1^3/4\\
  \hline 
 \endhead
 \hline
 \hline \multicolumn{4}{c}{{\textrm{Continued on next page}}} \\ \hline
 \endfoot
 
 \hline \hline
 \endlastfoot
 76=2^2\cdot 19 & 2& 1& \\
  80=2^4\cdot 5 &  2& 1, 2, {4}, \cancel{5}& 4\\
    80=2^4\cdot 5 & 4 & 1, \cancel{5} & 4\\
 84=2^2\cdot 3\cdot 7 & 2& 1,3 & \\
88=2^3\cdot 11 & 2 & 1, 2 & \\
90=2\cdot 3^2\cdot 5 & 3& 1,2, \cancel{5} & 1.5\\
 92=2^2\cdot 23 &2 & 1,2 & 
\end{longtable}
} 

\end{rem}



   

\begin{lem}\label{lem A2K even}
   Let $X$ be a Gorenstein canonical Fano $3$-fold.
   Suppose that $\Sing(X)$ contains no curve. Then for any $\mathbb{Q}$-Cartier Weil divisor $D$ on $X$, $(D^2\cdot K_X)$ is an even integer.  
\end{lem}
\begin{proof}
    By \cite{Reid1983}, a general element $S\in |-K_X|$ has at worst Du Val singularities, which means that $(X, S)$ is plt by the inversion of adjunction \cite[Theorem~5.50]{KollarMori1998}, which is equivalent to that $(X, S)$ is a canonical pair as $K_X+S\sim 0$ is Cartier. So $S\cap \Sing(X)$ contains only terminal singularities of $X$. As $D$ is Cartier near terminal singularities of $X$ by \cite[Lemma~5.1]{Kawamata1988}, $D|_S$ is a Cartier divisor on $S$. Hence $(D^2\cdot S)=(D|_S^2)$ is an even integer by the Riemann--Roch formula $\chi(S, D|_S)= \frac{1}{2}(D|_S^2)+2$, where $S$ is a K3 surface with Du Val singularities. 
\end{proof}

\section*{Acknowledgments} 
The authors would like to thank Yong Hu, Minyou Li and Y.~G.~Prokhorov for helpful discussions. They would also like to thank some anonymous referee for the useful comments. C.~Jiang was supported by National Key Research and Development Program of China (No. 2023YFA1010600, No. 2020YFA0713200) and NSFC for Innovative Research Groups (No. 12121001). C.~Jiang is a member of the Key Laboratory of Mathematics for Nonlinear Sciences, Fudan University. H.~Liu is supported by the National Key Research and Development Program of China (No. 2023YFA1009801) in part and NSFC (No. 12571048). J.~Liu is supported by the Youth Innovation Promotion Association CAS, the CAS Project for Young Scientists in Basic Research (No. YSBR-033) and the NSFC grant (No. 12288201).

\bibliography{JLLcanonical}

\end{document}